\documentclass[letterpaper, 10pt, conference]{ieeeconf}      

\IEEEoverridecommandlockouts                              
\usepackage{amsmath,graphicx,amsfonts,amssymb,epsfig,mathrsfs}
\usepackage{subcaption}
\usepackage{tikz}
\usetikzlibrary{shapes}
\usepackage{color}
\usepackage{slashbox}
\usepackage{multirow}
\usepackage{rotating}
\usepackage{graphicx}%
\usepackage{algorithm}
\usepackage{algpseudocode}
\usepackage{algorithmicx}
\usepackage{cite}
\usepackage{url}
\usepackage{framed}
\usepackage{bm}

\newtheorem{theorem}{Theorem}

\newtheorem{definition}{Definition}
\newtheorem{lemma}{Lemma}

\newtheorem{remark}{Remark}

\newcommand{\differential}{{\rm{d}}}

\title{\LARGE \bf
DeGroot-Friedkin Map in Opinion Dynamics is Mirror Descent
}

\author{Abhishek Halder
\thanks{Abhishek Halder is with the Department of Applied Mathematics, University of California, Santa Cruz, CA 95064, USA,
        {\tt\small{ahalder@ucsc.edu}}%
}}

\begin{document}

\maketitle
\thispagestyle{empty}
\pagestyle{empty}

\begin{abstract}
We provide a variational interpretation of the DeGroot-Friedkin map in opinion dynamics. Specifically, we show that the nonlinear dynamics for the DeGroot-Friedkin map can be viewed as mirror descent on the standard simplex with the associated Bregman divergence being equal to the generalized Kullback-Leibler divergence, i.e., an entropic mirror descent. Our results reveal that the DeGroot-Friedkin map elicits an individual's social power to be close to her social influence while minimizing the so called ``extropy" -- the entropy of the complimentary opinion. 
\end{abstract}


\section{Introduction}
The DeGroot-Freidkin map \cite{jia2015opinion} in opinion dynamics is a nonlinear recursion of the form 
\begin{eqnarray*}
\bm{x}(k+1) = \bm{f}\left(\bm{x}(k)\right), \quad \bm{f} : \Delta^{n-1} \mapsto \Delta^{n-1},
\end{eqnarray*}
where $\Delta^{n-1}:=\{\bm{x}\in\mathbb{R}^{n}_{\geq 0}\:\mid\:\bm{1}^{\top}\bm{x}=1\}$ denotes the standard simplex in $\mathbb{R}^{n}$, i.e., convex hull of the standard basis vectors $\bm{e}_{1}, \hdots, \bm{e}_{n}$ in $\mathbb{R}^{n}$. Let ${\rm{int}}\left(\Delta^{n-1}\right):=\{\bm{x}\in\mathbb{R}^{n}_{>0}\:\mid\:\bm{1}^{\top}\bm{x}=1\}$ denote the interior of this simplex. The state vector $\bm{x} = \left(x_{1}, \hdots, x_{n}\right)^{\top}$ models the self-weights of $n$ individuals exchanging opinions in a social network on a particular issue. The recursion index $k=0,1,...$ codifies a sequence of issues. The map $\bm{f}$ depends on a parameter vector $\bm{c}\in{\rm{int}}\left(\Delta^{n-1}\right)$, which is the Perron-Frobenius left eigenvector of an $n\times n$ row stochastic, zero-diagonal, irreducible\footnote[2]{A nonnegative matrix is irreducible if its associated digraph is strongly connected. The digraph associated with an $n\times n$ nonnegative matrix is constructed by adding a directed edge from node $i$ to $j$, where $i,j=1,...,n$, provided the $(i,j)$-th element of the matrix is positive.} matrix $\bm{C}$, typically referred to as the ``relative influence" or ``relative interaction matrix". As in the original DeGroot-Friedkin model, we will assume that the matrix $\bm{C}$ is constant. Under the stated structural assumptions on matrix $\bm{C}$, the vector $\bm{c}$  satisfies (see e.g., \cite[Lemma 2.3, part (i)]{jia2015opinion}) $0<c_{i}\leq 1/2$ for all $i=1,\hdots,n$. Intuitively, the elements of the matrix $\bm{C}$ model the relative influence of an individual's social network in her opinion, and they affect the opinion dynamics via vector $\bm{c}$. Thus, the DeGroot-Freidkin map describes how the self-weights of a group of individuals evolve over a sequence of issues accounting that social interactions influence opinion.

To ease notation, let $[n] := \{1,\hdots,n\}$. In the DeGroot-Freidkin model, the map $\bm{f}\left(\bm{x}(k)\right)$ is explicitly given by
\begin{equation}
\begin{aligned}
	\bm{f}(\bm{x}(k)) \!=\!\!\begin{cases}
 \bm{e}_{i} \qquad\qquad\qquad\qquad\text{if}\;\:\bm{x}(k)=\bm{e}_{i},\; i\in[n],\\
 	\displaystyle\frac{\bm{c} \oslash \left(\bm{1} - \bm{x}(k)\right)}{\bm{1}^{\top}\left(\bm{c} \oslash \left(\bm{1} - \bm{x}(k)\right)\right)} \;\text{otherwise},
\end{cases}	
\end{aligned}
\label{DeGrootFriedkinMap}
\end{equation}
for $k=0,1,\hdots$, where the symbol $\oslash$ denotes element-wise division, and $\bm{1}$ denotes the column vector of ones. The explicit form of the recursion appeared first in \cite[Lemma 2.2]{jia2015opinion}, and was proposed as a combination of the DeGroot model \cite{degroot1974reaching} and the Friedkin's model of reflected appraisal \cite{friedkin2011formal} in the evolution of social power (therefore, the name ``DeGroot-Friedkin model"). Various extensions of the basic DeGroot-Friedkin model have appeared in \cite{chen2018social,xu2015modified,askarzadeh2019stability}.

The convergence properties for the DeGroot-Friedkin map depend on whether the digraph associated with $\bm{C}$ has star topology or not. An $n$-vertex digraph has star topology if there exists a node $i\in[n]$, referred to as the ``center node", so that all directed edges of the digraph share the $i$-th vertex. In the opinion dynamics context, interpreting the vertices of the digraph as individuals, existence of star topology means that a single individual holds the social power to influence the opinion of the group.  

From (\ref{DeGrootFriedkinMap}), it is evident that the map $\bm{f}$ leaves the vertices of the simplex invariant, and hence $\{\bm{e}_{i}\:\mid\:i\in[n]\}$ are fixed points. For\footnote[3]{The DeGroot-Friedkin dynamics for the map (\ref{DeGrootFriedkinMap}) is degenerate for $n=2$ since in that case, $c_{1}=c_{2}=1/2$ and all points on the simplex are fixed points. In this paper, we thus consider $n\geq 3$.} $n\geq 3$, it is known \cite[Theorem 4.1]{jia2015opinion} that in addition to the simplex vertices, there exists a unique fixed point $\bm{x}^{*}\in{\rm{int}}\left(\Delta^{n-1}\right)$, provided the digraph associated with $\bm{C}$ does not have star topology. In that case, for all initial conditions $\bm{x}(0)\in\Delta^{n-1}\setminus\{\bm{e}_{1}, \hdots, \bm{e}_{n}\}$, the iterates $\bm{x}(k) \rightarrow \bm{x}^{*}$ as $k\rightarrow\infty$. On the other hand, if the digraph associated with $\bm{C}$ has star topology, then the simplex vertices $\{\bm{e}_{i}\:\mid\:i\in[n]\}$ are the only fixed points \cite[Lemma 3.2]{jia2015opinion}, and for all initial conditions $\bm{x}(0)\in\Delta^{n-1}\setminus\{\bm{e}_{1}, \hdots, \bm{e}_{n}\}$, the iterates $\bm{x}(k) \rightarrow \bm{e}_{i}$ as $k\rightarrow\infty$, where $i\in[n]$ is the index of the center node. In the rest of this paper, we will tacitly assume that the digraph associated with $\bm{C}$ does not have star topology, i.e., the map (\ref{DeGrootFriedkinMap}) admits $(n+1)$ fixed points $\{\bm{e}_{1}, ..., \bm{e}_{n}, \bm{x}^{*}\}$ where $\bm{x}^{*}\in{\rm{int}}\left(\Delta^{n-1}\right)$.

We can interpret the vertices of the simplex as ``autocratic" fixed points. The fixed point $\bm{x}^{*}$ in the interior of the simplex, is purely ``democratic" when it is equal to $\bm{1}/n$, which happens if and only if $\bm{C}$ is doubly stochastic. In general, the location of $\bm{x}^{*}\in{\rm{int}}\left(\Delta^{n-1}\right)$ depends on the parameter vector $\bm{c}$ (or equivalently, on the matrix $\bm{C}$).

The results mentioned in the preceding two paragraphs were derived in \cite{jia2015opinion} through Lyapunov analysis. The purpose of this paper is to present a variational interpretation of the opinion dynamics for the DeGroot-Friedkin map. Specifically, we show that the DeGroot-Friedkin map can be viewed as mirror descent of a convex function on the standard simplex with the associated Bregman divergence being equal to the generalized Kullback-Leibler divergence. On one hand, our development provides novel geometric insight for the opinion dynamics on standard simplex. On the other hand, it answers the natural question: what is the collective utility (i.e., ``social welfare") that the DeGroot-Friedkin map elicits over a given influence network?

This paper is organized as follows. Section \ref{SectionLabelMirrorDescent} provides an expository overview of mirror descent. Our main results are collected in Section \ref{SectionLabelMainResults}. Several implications of our variational interpretation are provided in Section \ref{SectionLabelRamification}. Section \ref{SectionLabelConclusions} concludes the paper.

\subsubsection*{Notations and preliminaries} We denote the entropy of a vector $\bm{p}\in\Delta^{n-1}$ as $H(\bm{p}):= -\sum_{i=1}^{n}p_{i}\log p_{i}$, and the Kullback-Leibler divergence between $\bm{p},\bm{q}\in\Delta^{n-1}$ as $D_{\rm{KL}}\left(\bm{p}\parallel\bm{q}\right):=\sum_{i=1}^{n}p_{i}\log\left(p_{i}/q_{i}\right)$. As is well-known, both $H(\bm{p})$ and $D_{\rm{KL}}\left(\bm{p}\parallel\bm{q}\right)$ are $\geq 0$. In this paper, the operators $\log(\cdot)$ and $\exp(\cdot)$ are to be understood element-wise. 
Given vectors $\bm{\alpha}=(\alpha_{1}, \hdots, \alpha_{n})^{\top}$ and $\bm{\beta}=(\beta_{1}, \hdots, \beta_{n})^{\top}$, we denote element-wise multiplication and division as $\bm{\alpha}\odot\bm{\beta} := (\alpha_{1}\beta_{1}, \hdots, \alpha_{n}\beta_{n})^{\top}$ and $\bm{\alpha}\oslash\bm{\beta} := (\alpha_{1}/\beta_{1}, \hdots, \alpha_{n}/\beta_{n})^{\top}$, respectively. By ${\rm{diag}}(\bm{\alpha})$ we mean a diagonal matrix with diagonal elements being equal to the entries of the vector $\bm{\alpha}$. The notations ${\rm{dom}}(\cdot)$ and ${\rm{range}}(\cdot)$ stand for domain and range of a function, respectively; ${\rm{cl}}(\cdot)$ stands for closure of an open set; ${\rm{bdy}}(\cdot)$ stands for boundary of a closed set. By closure ${\rm{cl}}(\cdot)$ of a function, we mean that its epigraph is a closed set. We use $\langle\cdot,\cdot\rangle$ to denote the standard Euclidean inner product. The Legendre-Fenchel conjugate \cite[Section 12]{rockaller1970} of  a function $\theta:\mathbb{R}^{n}\mapsto\mathbb{R}$, is $\theta^{*}:\mathbb{R}^{n}\mapsto\mathbb{R}$ given by
\[\theta^{*}(\bm{y}) = \underset{\bm{x}}{\sup}\: \{\langle\bm{y},\bm{x}\rangle - \theta(\bm{x})\}.\] 
We clarify here the notation that a function with superscript $*$ denotes the Legendre-Fenchel conjugate, while a vector with superscript $*$ denotes fixed point. The following property of the Legendre-Fenchel conjugate will be useful in this paper. Let $\tau(\bm{x}):=\theta(\bm{A}\bm{x}+\bm{b})$ where $\bm{A}\in\mathbb{R}^{n\times n}$ is nonsingular, and $\bm{b}\in\mathbb{R}^{n}$. Then 
\begin{eqnarray}
	\tau^{*}(\bm{y}) = \theta^{*}(\bm{A}^{-\top}\bm{y}) - \bm{b}^{\top}\bm{A}^{-\top}\bm{y}.
	\label{conjugatePostCompositionAffine}
\end{eqnarray}


\section{Mirror Descent}\label{SectionLabelMirrorDescent}
The mirror descent \cite{nemirovsky1983problem} is a generalization of the well-known projected gradient descent algorithm to account the pertinent geometry of the optimization problem. Recall that for solving a convex optimization problem of the form
\begin{eqnarray}
\underset{\bm{x}\in\mathcal{X}}{\text{minimize}}\:\phi(\bm{x}),
\label{projgradproblem}	
\end{eqnarray}
(i.e., $\phi(\cdot)$ is a convex function; $\mathcal{X}\subset\mathbb{R}^{n}$ is a convex set), the projected gradient descent with constant step-size $h>0$ is a two-step algorithm, given by
\begin{subequations}
\begin{align}
\bm{y}(k+1) &= \bm{x}(k) - h\bm{g}({k}), \label{step1projgraddescent}\\
\bm{x}(k+1) &= \text{proj}_{\mathcal{X}}^{\parallel\cdot\parallel_{2}}\left(\bm{y}(k+1)\right), \label{step2projgraddescent}
\end{align}
\label{projgraddescent}	
\end{subequations}
where the Euclidean projection operator $\text{proj}_{\mathcal{X}}^{\parallel\cdot\parallel_{2}}\left(\bm{\eta}\right) := \underset{\bm{\xi}\in\mathcal{X}}{\arg\min}\frac{1}{2}\parallel\bm{\xi} - \bm{\eta}\parallel_{2}^{2}$, the subgradient $\bm{g}({k})\in\partial\phi(\bm{x}(k))$ (the subdifferential), and $k=0,1,\hdots$. The mirror descent generalizes (\ref{projgraddescent}) by introducing the so-called \emph{mirror map} and its associated \emph{Bregman divergence} \cite{bregman1967relaxation}.
\begin{definition}\label{DefnMirrorMap}(\textbf{Mirror map})
Given the convex optimization problem (\ref{projgradproblem}), suppose $\psi(\cdot)$ is a differentiable, strictly convex function on an open convex set ${\rm{dom}}(\psi) \subseteq\mathbb{R}^{n}$, i.e., $\psi:{\rm{dom}}(\psi)\mapsto \mathbb{R}$, such that the constraint set $\mathcal{X}\subset{\rm{cl}}({\rm{dom}}(\psi))$, ${\rm{range}}(\nabla\psi)=\mathbb{R}^{n}$, and $\parallel \nabla\psi \parallel_{2}\:\rightarrow +\infty$ as $\bm{x} \rightarrow {\rm{bdy}}({\rm{cl}}({\rm{dom}}(\psi)))$. Then $\psi(\cdot)$ is called a mirror map.
\end{definition}

\begin{definition}\label{DefnBregmanDiv}(\textbf{Bregman divergence})
	Let $\psi(\cdot)$ be a mirror map as in Definition \ref{DefnMirrorMap}. The associated Bregman divergence $D_{\psi} : {\rm{dom}}(\psi)\times{\rm{dom}}(\psi)\mapsto\mathbb{R}_{\geq 0}$ is given by
\begin{eqnarray}
D_{\psi}(\bm{x},\bm{y}) := \psi(\bm{x}) - \bigg\{\psi(\bm{y}) +\langle\nabla\psi(\bm{y}),\bm{x}-\bm{y}\rangle\bigg\},
\label{BregmanDiv}	
\end{eqnarray}
and can be interpreted as the error at $\bm{x}$ due to first order Taylor approximation of $\psi(\cdot)$ about $\bm{y}$. In general, $D_{\psi}$ is non-symmetric and hence not a metric.    	
\end{definition}
With Definitions \ref{DefnMirrorMap} and \ref{DefnBregmanDiv} in place, the mirror descent algorithm associated with the mirror map $\psi(\cdot)$ is a modified version of (\ref{projgraddescent}), given by
\begin{subequations}
\begin{align}
\nabla\psi\left(\bm{y}(k+1)\right) &= \nabla\psi\left(\bm{x}(k)\right) - h\bm{g}({k}), \label{step1mirrordescent}\\
\bm{x}(k+1) &= \text{proj}_{\mathcal{X}}^{D_{\psi}}\left(\bm{y}(k+1)\right), \label{step2mirrordescent}
\end{align}
\label{mirrordescent}	
\end{subequations}
where the Bregman projection operator $\text{proj}_{\mathcal{X}}^{D_{\psi}}(\bm{\eta}) := \underset{\bm{\xi}\in\mathcal{X}}{\arg\min}D_{\psi}\left(\bm{\xi},\bm{\eta}\right)$, and $k=0,1,\hdots$.

The main insight behind (\ref{mirrordescent}) is the following. As the subgradient $\bm{g}({k})$ is an element of the dual space, the subtraction in (\ref{step1projgraddescent}) does not make sense unless the decision variable $\bm{x}$ in (\ref{projgradproblem}) belongs to a Hilbert space (since the dual space of a Hilbert space is isometrically isomorphic to the Hilbert space, thanks to the Riesz representation theorem \cite[Ch. 4]{rudin1987realandcomplex}). To circumvent this issue, (\ref{step1mirrordescent}) takes an element from the primal space to the dual space via $\bm{x}(k) \mapsto \nabla\psi(\bm{x}(k))$, performs the gradient update in the dual space, and maps back the updated value $\bm{y}(k+1)$ in the primal space. To ensure that $\bm{x}(k+1)$ be in the set $\mathcal{X}$, the Bregman projection is performed in (\ref{step2mirrordescent}). The choice of the mirror map is usually guided by the geometry of the set $\mathcal{X}$.

We note that (\ref{mirrordescent}) reduces to (\ref{projgraddescent}) by setting $\psi(\cdot) = \frac{1}{2}\parallel\cdot\parallel_{2}^{2}$ and ${\rm{dom}}(\psi) = \mathbb{R}^{n}$ (in this case, $D_{\psi}(\bm{x},\bm{y}) = \frac{1}{2}\parallel\bm{x}-\bm{y}\parallel_{2}^{2}$).

Of particular importance to us, is the choice $\psi(\bm{x}) \equiv -H(\bm{x})= \sum_{i=1}^{n}x_{i}\log x_{i}$ (the negative entropy), ${\rm{dom}}(\psi) = \mathbb{R}^{n}_{>0}$, resulting in 
\begin{eqnarray}
D_{\psi}(\bm{x},\bm{y}) = D_{\rm{KL}}\left(\bm{x}\parallel\bm{y}\right) - \bm{1}^{\top}\left(\bm{x}-\bm{y}\right),
\label{generalizedKL}	
\end{eqnarray}
the \emph{generalized Kullback-Leibler divergence}, named so because it equals $D_{\rm{KL}}\left(\bm{x}\parallel\bm{y}\right)$ when $\bm{x},\bm{y}\in\Delta^{n-1}$. In the opinion dynamics context, we set $\mathcal{X}\equiv\Delta^{n-1}\setminus\{\bm{e}_{1}, \hdots, \bm{e}_{n}\}$, and seek an equivalence between (\ref{DeGrootFriedkinMap}) and (\ref{projgradproblem}). Per Definition \ref{DefnMirrorMap}, notice that $\psi(\bm{x}) \equiv -H(\bm{x})$ is a valid mirror map since it is strictly convex and differentiable; furthermore, $\mathcal{X}\equiv\Delta^{n-1}\setminus\{\bm{e}_{1}, \hdots, \bm{e}_{n}\}\subset{\rm{cl}}({\rm{dom}}(\psi)) = \mathbb{R}^{n}_{\geq 0}$, ${\rm{range}}(\nabla\psi)={\rm{range}}(\bm{1}+\log\bm{x})=\mathbb{R}^{n}$, and $\parallel \nabla\psi \parallel_{2}\:\rightarrow +\infty$ as $\bm{x} \rightarrow {\rm{bdy}}({\rm{cl}}({\rm{dom}}(\psi)))$. Using (\ref{generalizedKL}), direct computation gives
\begin{eqnarray}
&&\text{proj}_{\Delta^{n-1}\setminus\{\bm{e}_{1}, \hdots, \bm{e}_{n}\}}^{D_{\psi}}(\bm{\eta}) := \underset{\bm{\xi}\in\Delta^{n-1}\setminus\{\bm{e}_{1}, \hdots, \bm{e}_{n}\}}{\arg\min}D_{\psi}\left(\bm{\xi},\bm{\eta}\right) \nonumber\\
&&= \displaystyle\frac{\bm{\eta}}{\bm{1}^{\top}\bm{\eta}}, \quad\text{for all}\quad\bm{\eta}\in{\rm{dom}}(\psi)=\mathbb{R}^{n}_{>0}.\label{generalizedKLprojection}	
\end{eqnarray}
Therefore, for the mirror map $\psi(\bm{x}) \equiv -H(\bm{x})$, the mirror descent algorithm (\ref{mirrordescent}) becomes
\begin{subequations}
\begin{align}
\bm{y}(k+1) &= \bm{x}(k) \odot \exp\left(- h\bm{g}({k})\right), \label{step1mirrordescentEntropy}\\
\bm{x}(k+1) &= \bm{y}(k+1)/\bm{1}^{\top}\bm{y}(k+1), \label{step2mirrordescentEntropy}
\end{align}
\label{mirrordescentEntropy}	
\end{subequations}
where $k=0,1,\hdots$. Notice that for $\bm{x}\in\Delta^{n-1}\setminus\{\bm{e}_{1}, \hdots, \bm{e}_{n}\}$, the map (\ref{DeGrootFriedkinMap}) is indeed in the form of a generalized Kullback-Leibler projection for a positive vector $\bm{c}\oslash(\bm{1}-\bm{x})$ onto the standard simplex. We next develop this correspondence between (\ref{DeGrootFriedkinMap}) and (\ref{mirrordescentEntropy}).


\section{Main Results}\label{SectionLabelMainResults}
In order to associate a variational problem of the form (\ref{projgradproblem}) with the DeGroot-Friedkin map, we transcribe (\ref{DeGrootFriedkinMap}) in the form (\ref{mirrordescentEntropy}) by setting
\begin{eqnarray}
\bm{c}\oslash(\bm{1}-\bm{x}) = \bm{x} \odot \exp\left(-h \bm{g}\right),\label{relatingOpt}	
\end{eqnarray}
where $\bm{g}\in\partial\phi(\bm{x})$, $\bm{x}\in\Delta^{n-1}\setminus\{\bm{e}_{1}, \hdots, \bm{e}_{n}\}$. Rearranging (\ref{relatingOpt}), we get
\begin{eqnarray}
\bm{g} = \displaystyle\frac{1}{h}\log\left(\bm{x}\odot(\bm{1}-\bm{x})\oslash\bm{c} \right),
\label{partialphipartialxi}	
\end{eqnarray}
which implies
\begin{subequations}
\begin{align}
&\phi(\bm{x}) = \frac{1}{h}\displaystyle\sum_{i=1}^{n}\int \log\left(\displaystyle\frac{x_{i}(1-x_{i})}{c_{i}} \right)\differential x_{i}\label{integralform}\\
&= \frac{1}{h}\left[\displaystyle\sum_{i=1}^{n}x_{i}\log\left(\frac{x_{i}}{c_{i}}\right)	- \left(1-x_{i}\right)\log\left(1-x_{i}\right) + (1-2x_{i})\right]\nonumber\\
&= \frac{1}{h}\left[D_{\rm{KL}}\left(\bm{x}\parallel\bm{c}\right) + H\left(\bm{1} - \bm{x}\right) + n - 2\right],\label{phiintermed}
\end{align}
\label{phi}
\end{subequations}
where we used $\sum_{i=1}^{n}x_{i}=1$. Since both $D_{\rm{KL}}(\cdot\parallel\cdot)$ and $H(\cdot)$ are nonnegative functions, hence from (\ref{phiintermed}) it follows that $\phi(\bm{x})\geq 0$ for all $\bm{x}\in\Delta^{n-1}$, $n\geq 3$. Furthermore, we have the following. 
\begin{lemma}\label{ThmConvexityphi}
The function $\phi\left(\cdot\right)$ in (\ref{phiintermed}) is strictly convex	 over $\Delta^{n-1}$ for $n\geq 3$.
\end{lemma}
\begin{proof}
Notice that
\begin{eqnarray}
D_{\rm{KL}}\left(\bm{x}\parallel\bm{c}\right) + H\left(\bm{1} - \bm{x}\right) = H(\bm{1}-\bm{x}) - H(\bm{x}) \nonumber\\
\qquad - (\log\bm{c})^{\top}\bm{x}.	\label{DKLplusH1-x}
\end{eqnarray}
The following non-trivial\footnote[4]{Notice that $H(\bm{1}-\bm{x}) - H(\bm{x})$ is \emph{not} convex on $[0,1]^{n}$. Yet, the function $H(\bm{1}-\bm{x}) - H(\bm{x})$ is ``simplex-convex".} result was proved in  \cite[Theorem 20]{vontobel2013bethe}: the function $H(\bm{1}-\bm{x}) - H(\bm{x})$ is strictly convex for $\bm{x}\in\Delta^{n-1}$, $n\geq 3$. Therefore, (\ref{DKLplusH1-x}) being the sum of a strictly convex and a linear function, is also strictly convex in $\bm{x}\in\Delta^{n-1}$. From (\ref{phiintermed}), the statement follows.
\end{proof}
\begin{remark}
In \cite{lad2015extropy}, the quantity $H(\bm{1}-\bm{x})$, $\bm{x}\in\Delta^{n-1}$, was referred to as the ``extropy", and was argued to be a complimentary concept of the entropy $H(\bm{x})$. Like entropy, the extropy is permutation invariant, achieves maximum at the uniform distribution $\bm{1}/n$, and minimum at the simplex vertices $\bm{e}_{i}$, $i\in[n]$. The quantities entropy and extropy coincide for $n=2$, but are different for $n\geq 3$ (see e.g., \cite[Section 2]{lad2015extropy}).
\end{remark}

Lemma \ref{ThmConvexityphi} and its preceding discussion reveal that computing the fixed point $\bm{x}^{*}\in{\rm{int}}(\Delta^{n-1})$ for the DeGroot-Friedkin map is equivalent to solving a convex optimization problem over $\Delta^{n-1}\setminus\{\bm{e}_{1}, \hdots, \bm{e}_{n}\}$. We summarize this in the following Theorem.
\begin{theorem}\label{PropositionCvxProblem}
For $n\geq 3$, and for a given $\bm{c}\in{\rm{int}}(\Delta^{n-1})$, let $\bm{x}^{*}\in{\rm{int}}\left(\Delta^{n-1}\right)$ be the non-autocratic fixed point of the DeGroot-Friedkin map (\ref{DeGrootFriedkinMap}). Then $\bm{x}^{*}$ equals
\begin{subequations}
\begin{align}
	&\underset{\bm{x}\in\Delta^{n-1}\setminus\{\bm{e}_{1}, \hdots, \bm{e}_{n}\}}{\arg\min}\:\phi(\bm{x})\\
&= \underset{\bm{x}\in\Delta^{n-1}\setminus\{\bm{e}_{1}, \hdots, \bm{e}_{n}\}}{\arg\min}\:\big\{D_{\rm{KL}}\left(\bm{x}\parallel\bm{c}\right) + H\left(\bm{1} - \bm{x}\right)\big\}.\label{explicitopt}
\end{align}	
\label{varrecursion}
\end{subequations}	
\end{theorem} 
\begin{proof}
The equivalence between the mirror descent (\ref{projgradproblem}) with $\mathcal{X}\equiv\Delta^{n-1}\setminus\{\bm{e}_{1}, \hdots, \bm{e}_{n}\}$ and the DeGroot-Friedkin map is due to (\ref{relatingOpt}), (\ref{partialphipartialxi}), (\ref{phi}). The convexity of the objective follows from Lemma \ref{ThmConvexityphi}. What remains to prove is that we must have $\bm{x}^{*}\in{\rm{int}}\left(\Delta^{n-1}\right)$, i.e., $\bm{x}^{*}$ cannot be on the boundary of the simplex. One way to show this is to observe from (\ref{DeGrootFriedkinMap}) that $x_{i}^{*} \propto c_{i}/(1-x_{i}^{*})$, i.e.,
\begin{eqnarray}
	 c_{i} \propto x_{i}^{*}(1-x_{i}^{*}) \:\Leftrightarrow\: \bm{c} = \displaystyle\frac{\bm{x}^{*}\odot\left(\bm{1}-\bm{x}^{*}\right)}{1 - \parallel\bm{x}^{*}\parallel_{2}^{2}}.
	\label{c2xstar}
\end{eqnarray}
Since $c_{i} > 0$ for all $i\in[n]$, from (\ref{c2xstar}) it follows that $x_{i}^{*} > 0$, i.e., $\bm{x}^{*}\in{\rm{int}}\left(\Delta^{n-1}\right)$. We will see below that (\ref{c2xstar}) can also be derived from the conditions of optimality for (\ref{explicitopt}).
\end{proof}

An immediate corollary of the above is that the fixed point $\bm{x}^{*}\in{\rm{int}}\left(\Delta^{n-1}\right)$ is unique and its basin of attraction is $\Delta^{n-1}\setminus\{\bm{e}_{1}, \hdots, \bm{e}_{n}\}$. These facts were established in \cite{jia2015opinion} via non-smooth Lyapunov analysis.

Problem (\ref{varrecursion}) minimizes the extropy (i.e., entropy of complimentary opinion) while staying close to the vector $\bm{c}$ in Kullback-Leibler sense. This can be interpreted as follows. The entries of $\bm{c}$, termed as ``eigenvector centrality scores", reveal social influence of an individual. The entries of the argmin $\bm{x}^{*}$ reveal the individual's social power. The Kullback-Leibler term in the objective in (\ref{varrecursion}) implies that an individual's social power tends to be close to her social influence. The extropy term promotes collective non-uniformity in complimentary opinion, i.e., penalizes the ``spread" of the complimentary opinion $(\bm{1}-\bm{x})$ for the group. The overall objective in (\ref{varrecursion}) encapsulates the combined effect of these two tendencies.

We now show that permutation on the entries of $\bm{c}$ leads to the same permutation on the entries of $\bm{x}^{*}$.

\begin{theorem}\label{Thmrankxstarequalsrankc}
For a given $\bm{c}\in{\rm{int}}(\Delta^{n-1})$, let $\bm{x}^{*}$ be the argmin for the convex problem (\ref{varrecursion}). For any $n\times n$ permutation matrix $\bm{P}$, let
\[\bm{y}^{*} := \underset{\bm{y}\in\Delta^{n-1}\setminus\{\bm{e}_{1}, \hdots, \bm{e}_{n}\}}{\arg\min}\:\big\{D_{\rm{KL}}\left(\bm{y}\parallel\bm{P}\bm{c}\right) + H\left(\bm{1} - \bm{y}\right)\big\}.\]
Then $\bm{y}^{*} = \bm{P}\bm{x}^{*}$.	
\end{theorem}
\begin{proof}
We start by noting that 
\begin{eqnarray}
D_{\rm{KL}}\left(\bm{y}\parallel\bm{P}\bm{c}\right) + H\left(\bm{1} - \bm{y}\right) = H(\bm{1} - \bm{y}) - H(\bm{y}) \nonumber\\
\qquad - \left(\log\left(\bm{Pc}\right)\right)^{\top}\bm{y},
\label{permutec}	
\end{eqnarray}
and that $\log(\bm{Pc})=\bm{P}\log\bm{c}$. Since $\bm{P}^{\top} = \bm{P}^{-1}$, hence letting $\bm{z} := \bm{P}^{-1}\bm{y}$, we can rewrite the right-hand-side of (\ref{permutec}) as
\begin{eqnarray}
H(\bm{1} - \bm{z}) - H(\bm{z}) - \left(\log\bm{c}\right)^{\top}\bm{z},
\label{RHSinz}	
\end{eqnarray}
where we have used that $H(\bm{1} - \bm{P}\bm{z}) - H(\bm{Pz}) = H(\bm{1} - \bm{z}) - H(\bm{z})$, as both entropy and extropy are permutation invariant. Therefore,
\begin{subequations}
\begin{align}
	\bm{y}^{*} &= \underset{\bm{Pz}\in\Delta^{n-1}\setminus\{\bm{e}_{1}, \hdots, \bm{e}_{n}\}}{\arg\min}\big\{H(\bm{1} - \bm{z}) - H(\bm{z}) - \left(\log\bm{c}\right)^{\top}\bm{z}\big\}\nonumber\\
&= \bm{P}\bm{x}^{*}.\nonumber
\end{align}	
\end{subequations}
This completes the proof.
\end{proof}

For problem (\ref{varrecursion}), since the objective is convex, and the constraint $\bm{1}^{\top}\bm{x} - 1 = 0$ is linear, strong duality holds. Let $\nu\in\mathbb{R}$ be the Lagrange multiplier associated with the constraint $\bm{1}^{\top}\bm{x} - 1 = 0$. The corresponding Lagrangian 
\begin{eqnarray}
L(\bm{x},\nu) = H(\bm{1} - \bm{x}) - H(\bm{x}) - (\log\bm{c})^{\top}\bm{x} + \nu(\bm{1}^{\top}\bm{x} - 1)
\label{Lagrangian}	
\end{eqnarray}
yields the following Karush-Kuhn-Tucker (KKT) conditions for the optimal pair $(\bm{x}^{*},\nu^{*})$:
\begin{subequations}
\begin{align}
&\nabla	_{\bm{x}}L\big\vert_{\bm{x}=\bm{x}^{*}} = \bm{0},\nonumber\\
&\Leftrightarrow (x_{i}^{*})^{2} - x_{i}^{*} + c_{i}\exp\left(-(\nu^{*} + 2)\right) = 0, \:\forall\:i\in[n],\label{gradcondn}\\
&\bm{1}^{\top}\bm{x}^{*} = 1.\label{primalfeasibility}
\end{align}
\label{KKT}
\end{subequations}
Summing (\ref{gradcondn}) over $i=1,...,n$, then using (\ref{primalfeasibility}) and $\bm{1}^{\top}\bm{c}=1$ reveals that
\begin{eqnarray}
\exp\left(-(\nu^{*} + 2)\right) \:=\:  1 - \parallel\bm{x}^{*}\parallel_{2}^{2}.
\label{1minusnormsq}	
\end{eqnarray}
Using (\ref{1minusnormsq}) to substitute for $\exp\left(-(\nu^{*} + 2)\right)$ in (\ref{gradcondn}) results the map $\bm{x}^{*}\mapsto\bm{c}$, given by
\begin{eqnarray}
\bm{c} = \displaystyle\frac{\bm{x}^{*}\odot\left(\bm{1}-\bm{x}^{*}\right)}{1 - \parallel\bm{x}^{*}\parallel_{2}^{2}},
\label{casfnofxstar}	
\end{eqnarray}
which is what we obtained in (\ref{c2xstar}).

At this point, recall that the matrix $\bm{C}$ being doubly stochastic is equivalent to $\bm{c}=\bm{1}/n$. We next use (\ref{casfnofxstar}) to further prove that $\bm{x}^{*}=\bm{1}/n$ if and only if $\bm{c}=\bm{1}/n$. 
\begin{theorem}\label{centroidThm}
Let $\bm{x}^{*}$ be the 	argmin for the convex problem (\ref{varrecursion}). Then, $\bm{x}^{*}=\bm{1}/n$ if and only if $\bm{c}=\bm{1}/n$.
\end{theorem}
\begin{proof}
For any $i\neq j$, using $c_{i} = c_{j} = 1/n$ in (\ref{casfnofxstar}), we obtain $x_{i}^{*}(1-x_{i}^{*}) = x_{j}^{*}(1-x_{j}^{*})$, since $0 < x_{i}^{*} < 1 \Rightarrow ||\bm{x}^{*}||^{2}\neq 1$, and $1 - \parallel\bm{x}^{*}\parallel_{2}^{2}$ = constant (from (\ref{1minusnormsq})). This gives $(x_{i}^{*} - x_{j}^{*})(1 - x_{i}^{*} - x_{j}^{*}) = 0$, for all $i,j=1,\hdots,n$. Notice that $x_{i}^{*} + x_{j}^{*} \neq 1$ since otherwise, remaining $(n-2)$ entries of the vector $\bm{x}^{*}$ would be zero, which contradicts the premise $\bm{x}^{*}\in{\rm{int}}(\Delta^{n-1})$. Hence $x_{i}^{*} = x_{j}^{*}$ for all $i,j=1,\hdots,n$. The condition $\sum_{i}x_{i}^{*}=1$ then yields $x_{i}^{*}=1/n$ for all $i\in[n]$. 

On the other hand,	directly substituting $x_{i}^{*}=1/n$ in (\ref{casfnofxstar}) results $c_{i}=1/n$ for all $i\in[n]$.
\end{proof}

\begin{figure}[tphb]
\centering
	\includegraphics[width=0.5\textwidth]{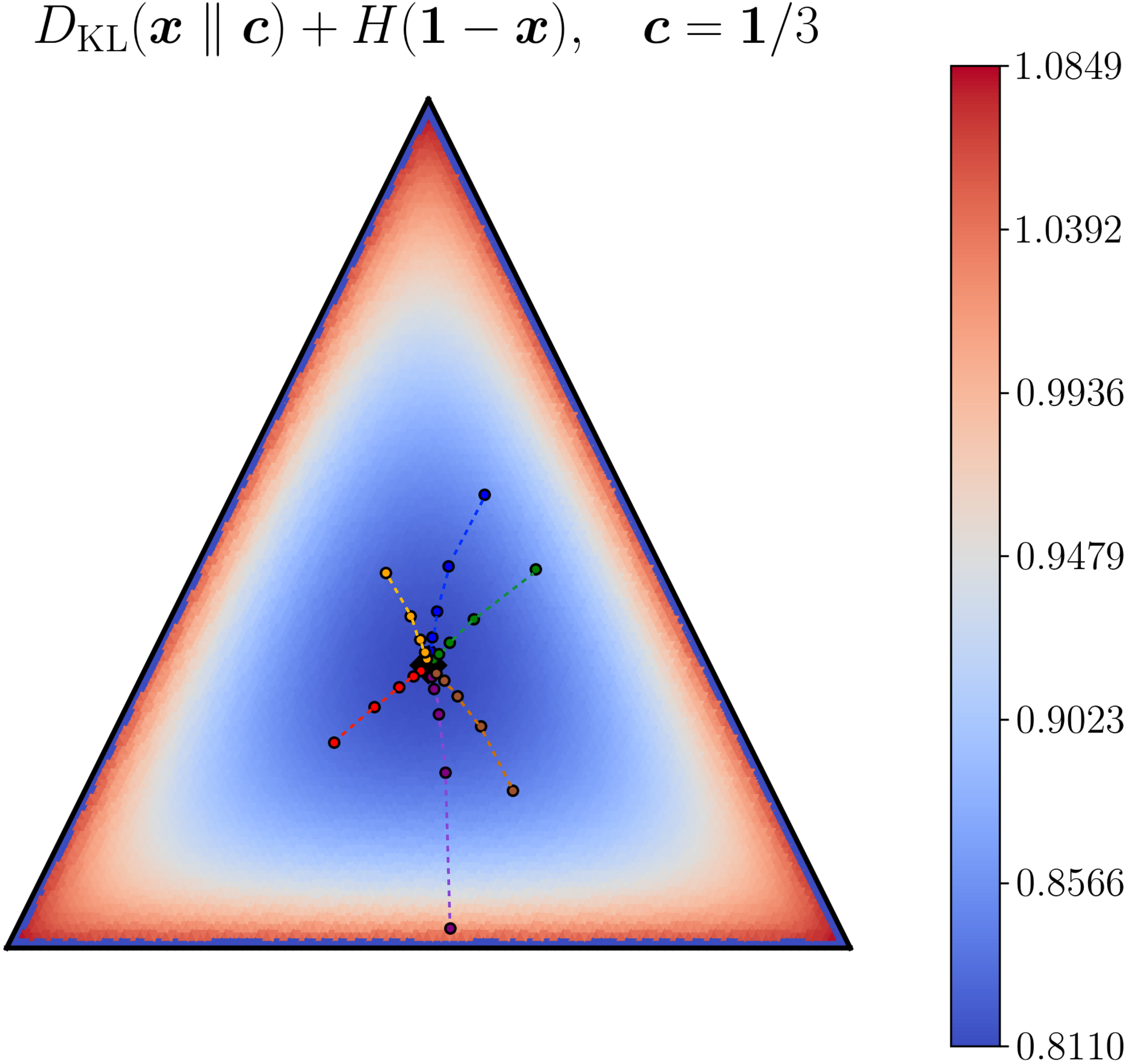}
\caption{The colormap of the convex objective function in (\ref{varrecursion}) for $n=3$ on the simplex $\Delta^{2}$ for $\bm{c}=\bm{1}/3$. In this figure, we also plot the fixed point $\bm{x}^{*}=\bm{1}/3$ (black diamond) and the first four iterates of six randomly chosen initial conditions (indicated by six different colored circles) for recursion (\ref{DeGrootFriedkinMap}), showing they converge to $\bm{x}^{*}$, which is the minimizer for the objective function in (\ref{varrecursion}).}
\label{UniformcFig}	
\end{figure}	
\begin{figure}[tphb]
\centering
	\includegraphics[width=0.5\textwidth]{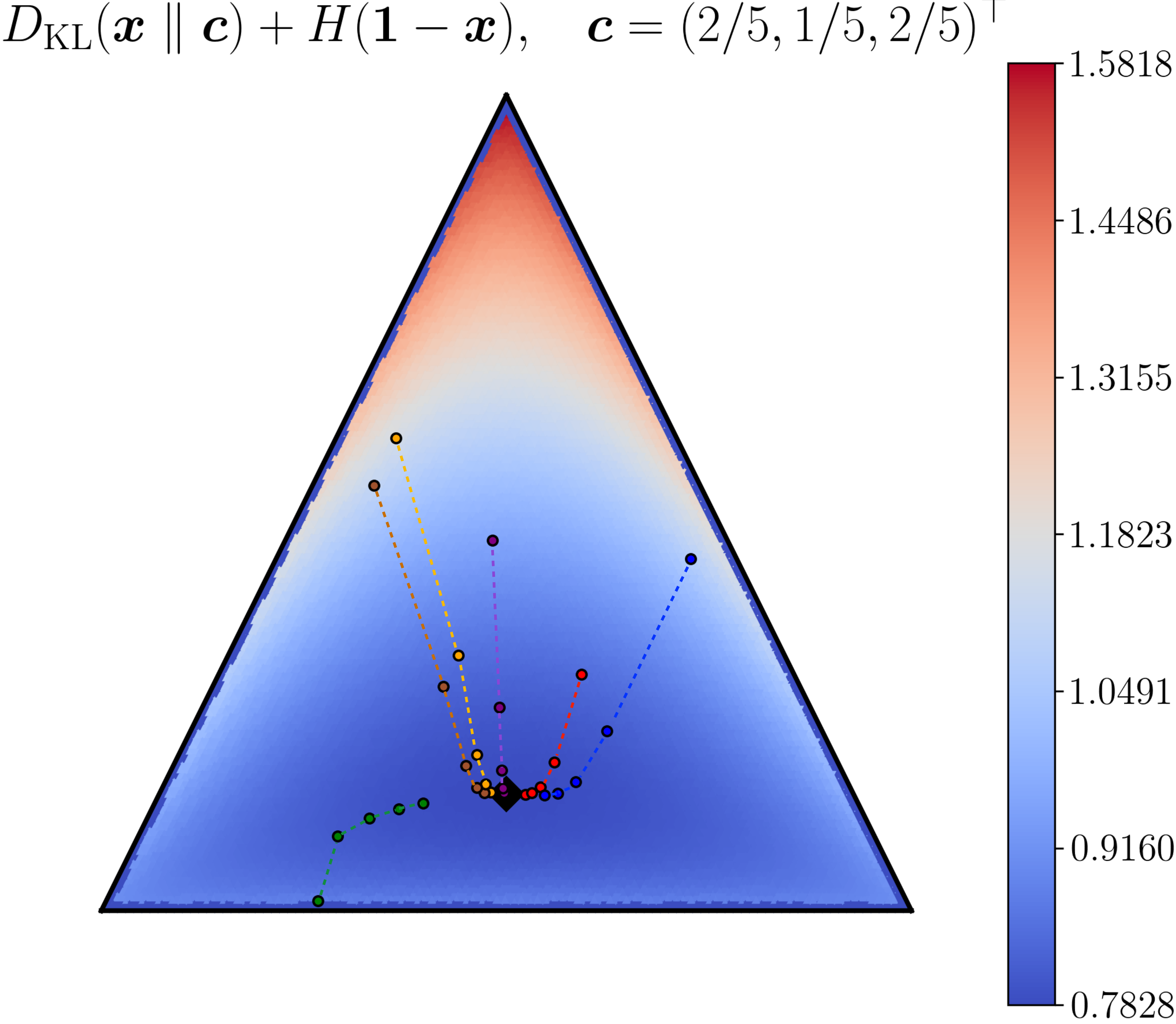}
\caption{The colormap of the convex objective function in (\ref{varrecursion}) for $n=3$ on the simplex $\Delta^{2}$ for $\bm{c}=(2/5, 1/5, 2/5)^{\top}$. In this figure, we also plot the fixed point $\bm{x}^{*}=(3/7, 1/7, 3/7)^{\top}$ (black diamond) and the first four iterates of six randomly chosen initial conditions (indicated by six different colored circles) for recursion (\ref{DeGrootFriedkinMap}), showing they converge to $\bm{x}^{*}$, which is the minimizer for the objective function in (\ref{varrecursion}).}
\label{NonuniformcFig}	
\end{figure}

\begin{remark}
Notice that for $\bm{c}=\bm{1}/n$, the problem (\ref{varrecursion}) reduces to computing the argmin of $H(\bm{1}-\bm{x}) - H(\bm{x})$ (due to (\ref{DKLplusH1-x})). Therefore, a corollary of Theorem \ref{centroidThm} is that $\bm{x}^{*}=\bm{1}/n$ is the minimizer of the convex function $H(\bm{1}-\bm{x}) - H(\bm{x})$.
\end{remark}

We now provide some numerical evidence to help visualize the development so far. In Fig. \ref{UniformcFig}, we plot the colormap of the objective function in (\ref{varrecursion}) for $n=3$ on the simplex $\Delta^{2}$ for $\bm{c}=\bm{1}/3$. This colormap suggests that the objective function achieves minimum at $\bm{x}=\bm{1}/3$, which is in accordance with Theorem \ref{centroidThm}. In the same figure, we overlay the fixed point $\bm{x}^{*}=\bm{1}/3$ (black diamond) and the first few iterates of six randomly chosen initial conditions (indicated by six different colored circles) for recursion (\ref{DeGrootFriedkinMap}), showing that all the iterates converge to $\bm{x}^{*}$, which is indeed the minimum of $D_{\rm{KL}}\left(\bm{x}\parallel\bm{c}\right) + H\left(\bm{1} - \bm{x}\right)$ over the simplex.

Likewise, in Fig. \ref{NonuniformcFig}, we plot the colormap of the objective function in (\ref{varrecursion}) on the simplex $\Delta^{2}$ for $\bm{c}=(2/5, 1/5, 2/5)^{\top}$. In this case, the fixed point $\bm{x}^{*} = (3/7, 1/7, 3/7)^{\top}$ (black diamond), which can be verified by direct substitution in (\ref{DeGrootFriedkinMap}). Again, in Fig. \ref{NonuniformcFig}, we overlay the first few iterates of six randomly chosen initial conditions (indicated by six different colored circles) for recursion (\ref{DeGrootFriedkinMap}) showing that all the iterates converge to $\bm{x}^{*}$, which is indeed the minimum of $D_{\rm{KL}}\left(\bm{x}\parallel\bm{c}\right) + H\left(\bm{1} - \bm{x}\right)$ over the simplex.

\section{Ramifications}\label{SectionLabelRamification}
Next, we collect some consequences which follow from our variational interpretation.

\subsection{Proximal Recursion}
A consequence of identifying $\phi(\cdot)$ with the mirror descent is that we can express the (transient) DeGroot-Friedkin iterates via proximal recursion:
\begin{equation}
	\bm{x}(k+1) = \!\!\!\underset{\bm{x}\in\Delta^{n-1}\setminus\{\bm{e}_{1},\hdots,\bm{e}_{n}\}}{\arg\min}\!\! D_{\rm{KL}}\left(\bm{x}\parallel\bm{x}(k)\right) \:+\: h \langle\bm{g}({k}),\bm{x}\rangle,
\label{proxrecursion}
\end{equation}
where $k=0,1,\hdots$, and $\bm{g}({k})\equiv\bm{g}(\bm{x}=\bm{x}(k))$ is given by (\ref{partialphipartialxi}). This proximal recursion perspective of mirror descent is due to \cite{beck2003mirror}. One can view (\ref{proxrecursion}) as minimizing the local linearization of $\phi$ while not being too far (in Kullback-Leibler sense) from the previous iterate. 

\subsection{Lagrange Dual Problem}
For theoretical completeness, we now derive the dual problem associated with the primal problem (\ref{explicitopt}). Since the constraint in (\ref{explicitopt}) is linear, we can derive the associated Lagrange dual problem using the Legendre-Fenchel conjugate (see e.g., \cite[p. 221, Section 5.1.6]{boyd2004convex}). Specifically, let $\phi_{0}(\bm{x}):=D_{\rm{KL}}(\bm{x}\parallel\bm{c}) + H(\bm{1}-\bm{x})$, $h_{1}(\bm{x}) := H(\bm{1}-\bm{x})$, $h_{2}(\bm{x}):=-H(\bm{x})$, and $h(\bm{x}):=h_{1}(\bm{x})+h_{2}(\bm{x})$. From (\ref{DKLplusH1-x}), $\phi_{0}(\bm{x}) = h(\bm{x}) + (-\log\bm{c})^{\top}\bm{x}$; its Legendre-Fenchel conjugate $\phi_{0}^{*}(\bm{y}) = h^{*}(\bm{y}+\log\bm{c})$. Thus, the Lagrange dual function $\zeta(\cdot)$ associated with the primal problem (\ref{explicitopt}) is
\begin{eqnarray}
\zeta(\nu) = -\nu -\phi_{0}^{*}(-\nu\bm{1}) = -\nu -h^{*}(-\nu\bm{1} + \log\bm{c}),
\label{DualFunction}	
\end{eqnarray}
where as before, $\nu\in\mathbb{R}$ is the Lagrange multiplier associated with the constraint $\bm{1}^{\top}\bm{x}-1=0$.

The Legendre-Fenchel conjugate $h^{*}(\cdot)$ in (\ref{DualFunction}) can be written as infimal convolution of $h_{1}^{*}$ and $h_{2}^{*}$, i.e.,
\begin{subequations}
	\begin{align}
		&h^{*}(\bm{y}) = {\rm{cl}}\left(\underset{\bm{u}+\bm{v}=\bm{y}}{\inf} h_{1}^{*}(\bm{u}) + h_{2}^{*}(\bm{v})\right)\\
		&= {\rm{cl}}\left(\underset{\bm{u}+\bm{v}=\bm{y}}{\inf} \bigg\{\displaystyle\sum_{i=1}^{n}\exp(u_{i}-1) + u_{i} + \exp(v_{i}-1)\bigg\}\right)\label{stepb}\\
		&= {\rm{cl}}\left(\underset{\bm{u}\in\mathbb{R}^{n}}{\inf} \bigg\{\displaystyle\sum_{i=1}^{n}\exp(u_{i}-1) + u_{i} + \exp(y_{i}-u_{i}-1)\bigg\}\right),\label{stepc}
	\end{align}
	\label{infconvolution}	
\end{subequations}
where we used (\ref{conjugatePostCompositionAffine}) to derive (\ref{stepb}). Performing the unconstrained minimization in (\ref{stepc}), we obtain
\begin{equation}
	h^{*}(\bm{y}) = {\rm{cl}}\left(\displaystyle\sum_{i=1}^{n}\bigg\{\rho(y_{i}) + 1 + \log\rho(y_{i}) + \frac{\exp(y_{i}-2)}{\rho(y_{i})}\bigg\}\right),
\label{hstar}
\end{equation}
where
\begin{eqnarray}
\rho(y_{i}) := -\frac{1}{2} + \displaystyle\sqrt{\frac{1}{4} + \exp(y_{i}-2)}.
\label{rhodef}	
\end{eqnarray}
Thus, the dual problem associated with the primal problem (\ref{explicitopt}) is $\underset{\nu\in\mathbb{R}}{\sup}\:\zeta(\nu)$, where $\zeta(\cdot)$ is given by (\ref{DualFunction}), and $h^{*}(\cdot)$ is given by (\ref{hstar}).

\subsection{Equivalent Natural Gradient Descent}
\emph{Natural gradient descent} \cite{amari1998natural} generalizes the standard gradient descent to a Riemannian manifold. Specifically, let $(\mathcal{M},\bm{M})$ be an $n$-dimensional Riemannian manifold with metric tensor $\bm{M}$. For an optimization problem of the form
\begin{eqnarray}
\underset{\bm{\mu}\in\mathcal{M}}{\text{minimize}}\:\varphi(\bm{\mu}),
\label{natgradproblem}	
\end{eqnarray}
the natural gradient descent on $(\mathcal{M},\bm{M})$ with fixed step size $h>0$, is given by
\begin{eqnarray}
\bm{\mu}(k+1) = \bm{\mu}(k) - h \bm{M(\bm{\mu}(k))}^{-1}\nabla_{\bm{\mu}}\varphi(\bm{\mu}(k)),
\label{natgraddescent}	
\end{eqnarray}
where $k=0,1,\hdots$, i.e., (\ref{natgraddescent}) steps in the steepest descent direction of $\varphi(\cdot)$ along the manifold $(\mathcal{M},\bm{M})$. We now exploit an equivalence established in \cite{raskutti2015} between the mirror descent with \emph{twice differentiable} mirror map $\psi$, and the natural gradient descent along the dual Riemannian manifold as follows. Since $\psi$ is strictly convex and twice differentiable, the Hessian $\nabla^{2}\psi$ is positive definite. Thus, the Bregman divergence $D_{\psi}: {\rm{dom}}(\psi)\times{\rm{dom}}(\psi)\mapsto\mathbb{R}_{\geq 0}$ induces the Riemannian manifold $({\rm{dom}}(\psi),\nabla^{2}\psi)$. Let $\mathcal{M}$ be the image of ${\rm{dom}}(\psi)$ under map $\nabla\psi$, i.e., $\bm{\mu} = \nabla_{\bm{x}}\psi(\bm{x})$, and let $\psi^{*}$ be the Legendre-Fenchel conjugate of $\psi$. Then, the dual Bregman divergence $D_{\psi^{*}}: \mathcal{M}\times\mathcal{M}\mapsto\mathbb{R}_{\geq 0}$ induces the Riemannian manifold $(\mathcal{M},\nabla^{2}\psi^{*})$. In \cite{raskutti2015}, $(\mathcal{M},\nabla^{2}\psi^{*})$ was interpreted as the dual Riemannian manifold of the primal Riemannian manifold $({\rm{dom}}(\psi),\nabla^{2}\psi)$. For the unconstrained case (see \cite[Theorem 1]{raskutti2015}), i.e., for $\mathcal{X}={\rm{cl}(\rm{dom})}(\psi)$ in (\ref{mirrordescent}), the mirror descent with mirror map $\psi$ is equivalent to the natural gradient descent (\ref{natgraddescent}) along the dual manifold $(\mathcal{M},\nabla^{2}\psi^{*})$. For the constrained case, i.e., for $\mathcal{X}\subset{\rm{cl}(\rm{dom})}(\psi)$ in (\ref{mirrordescent}), we modify (\ref{natgraddescent})
as \emph{projected} natural gradient descent, i.e.,
\begin{eqnarray}
\bm{\mu}(k+1) =\text{proj}_{\mathcal{X}^{*}}^{D_{\psi^{*}}}\left(\bm{\mu}(k) - h \bm{M(\bm{\mu}(k))}^{-1}\nabla_{\bm{\mu}}\varphi(\bm{\mu}(k))\right),
\label{constrainednatural}	
\end{eqnarray}
where $\text{proj}_{\mathcal{X}^{*}}^{D_{\psi^{*}}}(\bm{\lambda}) := \underset{\bm{\mu}\in\mathcal{X}^{*}}{\arg\min}\:D_{\psi^{*}}\left(\bm{\mu},\bm{\lambda}\right)$, and $\mathcal{X}^{*}:=\{\bm{\mu}\in\mathcal{M} \:\mid\: \bm{\mu}=\nabla_{\bm{x}}\psi(\bm{x}), \bm{x}\in \mathcal{X}\}$. Thus, in general, the mirror descent (\ref{mirrordescent}) is equivalent to the projected natural gradient descent (\ref{constrainednatural}).

For our particular instance (\ref{mirrordescentEntropy}), $\psi(\bm{x}) = -H(\bm{x})$ is twice differentiable. Furthermore, $\bm{\mu} = \nabla_{\bm{x}}\psi(\bm{x}) = \bm{1} + \log\bm{x}$, $\psi^{*}(\bm{\mu}) = \bm{1}^{\top}\exp(\bm{\mu}-\bm{1})$, $\nabla_{\bm{\mu}}^{2}\psi^{*}(\bm{\mu}) = {\rm{diag}}(\exp(\bm{\mu}-\bm{1}))$, and (\ref{constrainednatural}) becomes
\begin{align}
\bm{\mu}(k+1) = &\text{proj}_{\mathcal{X}^{*}}^{D_{\psi^{*}}}\left(\bm{\mu}(k) \:- \right.\nonumber\\
&\left. h\left(\nabla_{\bm{\mu}}^{2}\psi^{*}(\bm{\mu})\right)^{-1} \nabla_{\bm{\mu}}\phi(\nabla_{\bm{\mu}}\psi^{*}(\bm{\mu}))\right),
\label{ournatgraddescent}
\end{align}
where $\phi$ is given by (\ref{phiintermed}). Substituting  $\psi^{*}(\bm{\mu}) = \bm{1}^{\top}\exp(\bm{\mu}-\bm{1})$ in (\ref{ournatgraddescent}) gives
\begin{eqnarray}
\bm{\mu}(k+1) = \text{proj}_{\mathcal{X}^{*}}^{D_{\psi^{*}}}\left(\log\left(\bm{c} \oslash \left(\exp(\bm{1}) - \exp(\bm{\mu}(k))\right)\right)\right).	
\label{intermednatgrad}
\end{eqnarray}
The Lemma below helps in computing the projection in (\ref{intermednatgrad}).
\begin{lemma}\label{DualBregmanLemma}
Let $\bm{\mu}=\nabla_{\bm{x}}\psi(\bm{x})$, and $\bm{\lambda}=\nabla_{\bm{y}}\psi(\bm{y})$. Then $D_{\psi^{*}}(\bm{\mu},\bm{\lambda}) = D_{\psi}(\bm{y},\bm{x})$.	
\end{lemma}
\begin{proof}
The proof follows from the definitions of the Bregman divergence and the Legendre-Fenchel conjugate.	
\end{proof}
In (\ref{intermednatgrad}), let $\bm{\lambda}:=\log\left(\bm{c} \oslash \left(\exp(\bm{1}) - \exp(\bm{\mu}(k))\right)\right)$, $\bm{\mu}(k+1)=\bm{1}+\log\widetilde{\bm{x}}$, and $\bm{\lambda} = \bm{1} + \log\bm{y}$. Thanks to Lemma \ref{DualBregmanLemma}, 
\begin{eqnarray}
\widetilde{\bm{x}} = \underset{\bm{x}\in\mathcal{X}}{\arg\min}\:D_{\psi}(\bm{y},\bm{x}),
\label{Mprojection}	
\end{eqnarray}
where $\mathcal{X}\equiv\Delta^{n-1}\setminus\{\bm{e}_{1}, \hdots, \bm{e}_{n}\}$, and $D_{\psi}$ is given by (\ref{generalizedKL}). Direct calculation yields $\widetilde{\bm{x}} = \bm{y}/\bm{1}^{\top}\bm{y}$. Therefore, (\ref{intermednatgrad}) gives
\begin{align}
	\!\!\!\bm{\mu}(k+1) \!=\! \bm{1}\!+\!\log\widetilde{\bm{x}} &= \bm{1}\!+\!\log\left(\bm{y}/\bm{1}^{\top}\bm{y}\right)	\nonumber\\
&= \log\left(\exp(\bm{\lambda})/\bm{1}^{\top}\exp(\bm{\lambda}-\bm{1})\right).
\label{FinalSimplification}
\end{align}
Substituting $\bm{\lambda}=\log\left(\bm{c} \oslash \left(\exp(\bm{1}) - \exp(\bm{\mu}(k))\right)\right)$ back in (\ref{FinalSimplification}) followed by algebraic simplification results
\begin{eqnarray}
\exp(\bm{\mu}(k+1)-\bm{1}) = \displaystyle\frac{\bm{c} \oslash \left(\bm{1} - \exp(\bm{\mu}(k)-\bm{1})\right)}{\bm{1}^{\top}\left(\bm{c} \oslash \left(\bm{1} - \exp(\bm{\mu}(k)-\bm{1})\right)\right)}.
\label{natgradrecursion}	
\end{eqnarray}
Since $\bm{\mu} = \bm{1} + \log\bm{x}$, hence the natural gradient recursion (\ref{natgradrecursion}) is exactly the DeGroot-Friedkin map (\ref{DeGrootFriedkinMap}).

\begin{remark}
The equivalence between the mirror descent (\ref{mirrordescentEntropy}) and the natural gradient descent (\ref{ournatgraddescent}) allows us to interpret the DeGroot-Friedkin map as steepest descent of $\phi(\nabla_{\bm{\mu}}\psi^{*}(\bm{\mu})) = D_{\rm{KL}}(\exp(\bm{\mu}-\bm{1})\parallel\bm{c}) + H(\bm{1} - \exp(\bm{\mu}-\bm{1}))$ along the manifold $(\mathcal{M},{\rm{diag}}(\exp(\bm{\mu}-\bm{1})))$, where $\mathcal{M}$ is the image of $\mathcal{X}\equiv\Delta^{n-1}\setminus\{\bm{e}_{1},\hdots,\bm{e}_{n}\}\subset{\rm{dom}}(\psi)\equiv\mathbb{R}^{n}_{>0}$ under map $\nabla_{\bm{x}}\psi = \bm{1} + \log\bm{x}$. 	In other words, the steepest descent occurs on the space of (shifted) log-likelihood.
\end{remark}

\begin{remark}
	In the information geometry literature \cite{csiszar1975divergence,amari2007methods}, (\ref{Mprojection}) is called the \emph{moment or M-projection} while (\ref{generalizedKLprojection}) is called the \emph{information or I-projection}. For arbitrary $\mathcal{X}$, (\ref{generalizedKLprojection}) and (\ref{Mprojection}) are not equal in general. 
\end{remark}


\section{Conclusions}\label{SectionLabelConclusions}
The DeGroot-Friedkin model for opinion dynamics describes the evolution of social power as a group of individuals discuss a sequence of issues over a network. We show that the DeGroot-Friedkin dynamics has a variational interpretation, i.e., the group of individuals collectively minimize a convex function of the opinions or self-weights. In particular, we prove that the nonlinear recursion associated with the DeGroot-Friedkin map can be viewed as entropic mirror descent over the standard simplex. Our variational formulation recovers known properties of the DeGroot-Friedkin map which were proved earlier via non-smooth Lyapunov analysis. Furthermore, the mirror descent framework reveals new interpretations of the DeGroot-Friedkin dynamics -- as a proximal recursion, and as a steepest descent on the space of log-likelihood. We hope that our results will motivate further investigations of opinion dynamics models from a variational perspective. Future work will involve numerical results based on real social network datasets.




\end{document}